\documentclass{amsart}
\usepackage{times,mathrsfs,array,amssymb,pb-diagram,multirow}

\newcounter{lemma}
\newtheorem{Theorem}{Theorem}
\newtheorem{Lemma}[lemma]{Lemma}
\newtheorem{Corollary}[lemma]{Corollary}

\theoremstyle{definition}

\newtheorem{Remark}[lemma]{Remark}

\def\Gal{\mathrm{Gal}}

\def\tor{\mathrm{tor}}

\def\C{\mathbb C}
\def\Q{\mathbb Q}

\def\Z{\mathbb Z}
\def\sC{\mathscr C}
\def\sU{\mathscr U}
\def\mod{\  \mathrm{mod}\ }

\def\JS#1#2{\left(\frac{#1}{#2}\right)}

\def\SL{\mathrm{SL}}
\def\wt{\widetilde}
\def\M#1#2#3#4{\begin{pmatrix}#1&#2\\#3&#4\end{pmatrix}}
\def\SM#1#2#3#4{\left(\begin{smallmatrix}#1&#2\\#3&#4\end{smallmatrix}
  \right)}
\def\div{\mathrm{div}\,}
\def\Div{\mathrm{Div}}

\def\->{\rightarrow}
\def\<->{\leftrightarrow}

\def\~{\widetilde}
\newcommand{\tabcaption}{\def\@captype{table}\caption}

\begin{document}
\title[Modular units and cuspidal divisor classes on $X_0(N)$]{Modular
  units and cuspidal divisor classes on $X_0(n^2M)$ with $n|24$ and
  $M$ squarefree}

\author{Liquan Wang}
\address{School of Mathematics and Statistics, Wuhan University, Wuhan
  430072, People's Republic of China}
\email{wanglq@whu.edu.cn}

\author{Yifan Yang}
\address{Department of Mathematics, 
National Taiwan University and National Center for Theoretical
Sciences, Taipei, Taiwan 10617}
\email{yangyifan@ntu.edu.tw}

\begin{abstract} For a positive integer $N$, let $\mathscr C(N)$ be
  the subgroup of $J_0(N)$ generated by the equivalence classes of
  cuspidal divisors of degree $0$ and $\mathscr C(N)(\mathbb
  Q):=\mathscr C(N)\cap J_0(N)(\mathbb Q)$ be its $\mathbb Q$-rational
  subgroup. Let also
  $\mathscr C_{\mathbb Q}(N)$ be the subgroup of $\mathscr
  C(N)(\mathbb Q)$ generated by $\mathbb Q$-rational
  cuspidal divisors. We prove that when $N=n^2M$ for some integer $n$
  dividing $24$ and some squarefree integer $M$, the two groups
  $\mathscr C(N)(\mathbb Q)$ and $\mathscr C_{\mathbb Q}(N)$ are
  equal. To achieve this, we show that all modular units on $X_0(N)$
  on such $N$ are products of functions of the form $\eta(m\tau+k/h)$,
  $mh^2|N$ and $k\in\mathbb Z$ and determine the necessary and
  sufficient conditions for products of such functions to be modular
  units on $X_0(N)$.
\end{abstract}

\thanks{This work was motivated by a remark about the equality between
  $\sC_\Q(N)$ and $\sC(N)(\Q)$ made by Hwajong Yoo in his talk given
  at Workshop on Eisenstein Ideals and Iwasawa Theory, Beijing, June
  17--22, 2019. The second author would like to thank the organizers,
  Emmanuel Lecouturier in particular, for inviting him to this
  wonderful workshop. He enjoyed discussions with the participants of
  the workshop, including Yuan Ren, Ken Ribet, Takao Yamazaki, and
  Hwajong Yoo.
  The authors would also like to thank the anonymous referees for many
valuable comments that greatly improve the exposition of the paper.}


\maketitle

\section{Introduction}
Let $N$ be a positive integer. In this note, we are primarily
concerned with modular units on the modular curve $X_0(N)$, i.e.,
modular functions on $X_0(N)$ whose divisors are supported on cusps,
and the cuspidal subgroup of the Jacobian variety $J_0(N)$ of
$X_0(N)$.

To describe relavent results in literature, we recall that a divisor
$D\in\Div(X_0(N))$ is said to be \emph{cuspidal} if its support lies
on cusps of $X_0(N)$. We let $\sC(N)$ be the subgroup of $J_0(N)$
generated by the equivalence classes of cuspidal divisors of
degree $0$ on $X_0(N)$, and 
refer to it as the \emph{cuspidal subgroup} of $J_0(N)$. By a
well-known result of Manin and Drinfeld \cite{Manin}, $\sC(N)$ is
contained in the torsion subgroup $J_0(N)_\tor$ of $J_0(N)$. Let also
$$
\sC(N)(\Q):=\sC(N)\cap J_0(N)(\Q)
$$
and $\sC_\Q(N)$ be the subgroup of $\sC(N)$ generated by
$\Q$-rational cuspidal divisors of degree $0$ on $X_0(N)$. (Here
we say a cuspidal divisor $D$ is $\Q$-rational if $\sigma(D)=D$ for
all $\sigma\in\Gal(\overline\Q/\Q)$.) Since the study of $\sC(N)$ is
equivalent to the study of modular units, we introduce the following
two groups:
$$
\sU(N):=\{\text{modular units on }X_0(N)\}/\C^\times
$$
and
$$
\sU_\Q(N):=\{f\in\sU(N):~\div f\text{ is }\Q
\text{-rational}\}/\C^\times.
$$

Now we have the inclusions of three groups
\begin{equation} \label{equation: inclusions}
\sC_\Q(N)\subseteq\sC(N)(\Q)\subseteq J_0(N)(\Q)_\tor.
\end{equation}
When the level $N$ is $2^rM$ for some odd squarefree integer $M$ and
some nonnegative integer $r\le3$, every cusp of $X_0(N)$ is
$\Q$-rational and hence $\sC_\Q(N)=\sC(N)(\Q)$. However, as pointed
out by Ken Ribet and other mathematicians, it is not clear a priori
whether $\sC_\Q(N)$ and $\sC(N)(\Q)$ are equal in general. It could
happen that even though $D$ itself is not a $\Q$-rational cuspidal
divisor, one still has $\sigma(D)\sim D$ for all
$\sigma\in\Gal(\overline\Q/\Q)$ so that $D\in\sC(N)(\Q)$. For example,
for $N=25$, the cusps $a/5$, $a=1,\ldots4$, are defined over
$\Q(e^{2\pi i/5})$ and they are Galois conjugates of each other, but
since $X_0(25)$ has genus $0$ and $J_0(25)$ is trivial, any cuspidal
divisor class $(a/5)-(b/5)$ is a $\Q$-rational point of the (trivial)
Jacobian. In fact, it took quite an effort in \cite[Pages
1268--1273]{Takagi-J2p} to prove that in the case of $X_1(2p)$, $p$ a
prime, two analogously defined groups are indeed equal.

For the second inclusion in \eqref{equation: inclusions}, Ogg
\cite{Ogg-survey} conjectured and later Mazur \cite{Mazur} proved that
in the case $N=p$ is a prime, one has $J_0(p)(\Q)_\tor=\sC_\Q(p)$ and
the group is cyclic of order $(p-1)/(p-1,12)$ generated by the class
of $(0)-(\infty)$. Since then, many mathematicians have tried to
extend Mazur's theorem to general cases. Here we list some known
results in literature about $\sC_\Q(N)$ and $J_0(N)(\Q)_\tor$.
\begin{enumerate}
\item[(a)] Lorenzini \cite{Lorenzini} showed that when $N=p^n$ is a
  prime power with $p\ge 5$ and $p\not\equiv 11\mod 12$, one has
  $$
  \sC_\Q(p^n)\otimes\Z[1/2p]\simeq J_0(p^n)(\Q)_\tor
  \otimes\Z[1/2p].
  $$
\item[(b)] Assume that $N=p^n$ is a prime power with $p\ge 5$. Ling
  \cite{Ling} computed the cardinality and the structure of
  $\sC_\Q(N)$ and proved that
  $$
  \sC_\Q(p^n)\otimes\Z[1/6p]\simeq J_0(p^n)(\Q)_\tor
  \otimes\Z[1/6p].
  $$
  One key property used in the proof is the fact that all modular
  units in $\sU_\Q(N)$ are products of the Dedekind eta functions.
  (This follows from either \cite[Theorem 1]{Newman2} or
  \cite[Proposition 3.2.1]{Ligozat}.)
  Later on, Yamazaki and Yang \cite{Yamazaki-Yang} obtained a basis
  for $\sU_\Q(p^n)$, $p\ge 5$, using Ling's cuspidal class number
  formula.
\item[(c)] Assume that $N$ is squarefree. Takagi \cite{Takagi} also
  used the fact that all modular units on $X_0(N)$ are products of
  the Dedekind eta functions to compute the cuspidal class number and
  described the structure of $\sC(N)(=\sC_\Q(N))$. Note that the
  special case where $N$ is a product of two primes was treated
  earlier in \cite{Chua-Ling}.
\item[(d)] Again, assume that $N$ is squarefree. Ohta \cite{Ohta}
  showed that
  $$
  \sC(N)\otimes\Z[1/6]\simeq J_0(N)(\Q)_\tor\otimes\Z[1/6],
  $$
  and in addition, if $3\nmid N$, then
  \begin{equation} \label{equation: Ohta 2}
  \sC(N)\otimes\Z[1/2]\simeq J_0(N)(\Q)_\tor\otimes\Z[1/2].
  \end{equation}
  In \cite{Yoo}, Yoo showed that if $p$ is
  a prime greater than $3$ such that either $p\not\equiv 1\mod 9$ or
  $3^{(p-1)/3}\not\equiv 1\mod p$, then \eqref{equation: Ohta 2} also
  holds for $N=3p$.
\item[(e)] Ren \cite{Ren} proved that for any positive integer $N$,
  $$
  J_0(N)(\Q)_\tor\otimes\Z[1/N']\simeq 0,
  $$
  where $N'=6N\prod_{p|N}(p^2-1)$. That is, for a prime $p'$, the
  $p'$-primary part of $J_0(N)(\Q)_\tor$ is trivial unless $p'$
  divides $N'$.
\item[(f)] In a very recent preprint \cite{Yoo2}, Yoo completely
  determined the structure of $\sC_\Q(N)$ for all $N$.
\end{enumerate}
Note that the cuspidal divisor subgroups of $J_1(N)$, the Jacobian of
$X_1(N)$, have also been studied by many authors. See, for instance,
\cite{Hazama,Sun,Takagi-X1p,Takagi-X1pm,Takagi-X12p,Takagi-J2p,Yang2,Yu}.

In this note, we will consider the case where $N$ is of the form
$N=n^2M$ for some integer $n$ dividing $24$ and squarefree $M$ ($2|M$
and $3|M$ permitted). The primary reason for considering such levels
is that modular units in these cases can still be expressed in terms
of the Dedekind eta functions. The key observation is that if $h|24$,
then $\eta(m\tau+k/h)$ is modular on $\Gamma_0(h^2m)$ in
the sense that
$$
\eta(m\gamma\tau+k/h)=\epsilon\sqrt{\frac{c\tau+d}i}
\eta(m\tau+k/h)
$$
for all $\gamma=\SM abcd\in\Gamma_0(h^2m)$ for some root of unity
$\epsilon$ depending on $\gamma$ (see Lemma \ref{lemma: eta mk}
below). Our approaches and results rely crucially on this observation
and cannot be extended to the cases $n\nmid 24$.

Throughout the remainder of the paper, we assume that $N=n^2M$ with
$n|24$ and $M$ squarefree. For a positive divisor $m$ of $N$, let
$h=h(m)$ be the largest integer such that $mh^2|N$, and for an integer
$k$, we define
$$
\eta_{m,k}(\tau):=\eta(m\tau+k/h)
=e^{2\pi i(m\tau+k/h)/24}\prod_{\ell=1}^\infty
\left(1-e^{2\pi i\ell k/h}q^{m\ell}\right), \quad q=e^{2\pi i\tau}.
$$
Our first main result gives the necessary and sufficient conditions
for a product of $\eta_{m,k}$ to be a modular function on $X_0(N)$.
The conditions are reminiscent of a well-known
criterion (see \cite[Proposition 3.2.8]{Ligozat}) for
$\prod_{d|N}\eta(d\tau)^{e_d}$ to be a modular function on $X_0(N)$.
Note that since $\eta_{m,k}$ and $\eta_{m,k+h}$ differ only by a root
of unity, we may assume that $k$ is in the range $0\le k\le h(m)-1$.

\begin{Theorem} \label{theorem: modular units}
  Let $N=n^2M$ with $n|24$ and $M$ squarefree and let $\eta_{m,k}$ and
  $h(m)$ be defined as above. Then a product of the form
  \begin{equation} \label{equation: product of eta mk}
    \prod_{m|N}\prod_{k=0}^{h(m)-1}\eta_{m,k}^{e_{m,k}}, \quad
    e_{m,k}\in\Z,
  \end{equation}
  is a modular function on $X_0(N)$ if and only if the
  integers $e_{m,k}$ satisfy the following conditions:
  \begin{enumerate}
  \item[(a)] $\displaystyle\sum_{m,k}e_{m,k}=0$,
  \item[(b)] $\displaystyle\sum_{m,k}e_{m,k}m\equiv 0\mod 24$,
  \item[(c)]
    $\displaystyle\sum_{m,k}e_{m,k}\frac{N(h(m),k)^2}{mh(m)^2}\equiv
    0\mod 24$,
  \item[(d)]
    \begin{enumerate}
    \item[(i)] In the case $n=3$ is odd,
    $$
    \sum_{m,k}e_{m,k}k\equiv 0\mod 3
    $$
    and
    \begin{equation} \label{equation: rational square}
    \prod_{m,k}m^{e_{m,k}}
    \end{equation}
    is the square of a rational number.
    \item[(ii)] In the case $n$ is even,
      \begin{equation} \label{equation: mod 2}
      \sum_{m,k}e_{m,k}\left(\frac{kn}{h(m)}
        +\frac n2\mathrm{ord}_2(m)\right)\equiv 0\mod n
      \end{equation}
      and the odd part of \eqref{equation: rational square} is the
      square of a rational number, where $\mathrm{ord}_2(m)$ denotes
      the $2$-adic valuation of $m$.
    \end{enumerate}
  \end{enumerate}
  Here the summation $\sum_{m,k}$ and the product $\prod_{m,k}$ are
  understood to be over pairs $(m,k)$ of integers with $m|N$ and $0\le
  k\le h(m)-1$.
\end{Theorem}

\begin{Remark} Note that the first three conditions represent the
  requirements that the weight is $0$, and the orders of the function
  at the cusps $\infty$ and $0$ are integers, respectively. See
  Corollary \ref{corollary: orders at infinity and 0} below.
\end{Remark}

Noticing that the number of such functions $\eta_{m,k}$ exceeds the
rank of $\sU(N)$, in the next theorem, we shall find a subset of such
functions so that every modular unit is uniquely expressed as a
product of functions from this subset.

\begin{Theorem} \label{theorem: modular units 2}
  Let $N=n^2M$ with $n|24$ and $M$ squarefree. Then
  every modular unit on $X_0(N)$ can be uniquely expressed as
  $$
  c\prod_{m|N}\prod_{k=0}^{\phi(h(m))-1}\eta_{m,k}^{e_{m,k}}
  $$
  for some nonzero complex numbers $c$ and integers $e_{m,k}$
  satisfying the conditions in Theorem \ref{theorem: modular units},
  where $\phi$ is Euler's totient function.
\end{Theorem}

\begin{Remark} The interested reader may use the following relations
  \begin{equation} \label{equation: eta relations}
    \begin{split}
      \eta(\tau+1/2)&=e^{2\pi i/48}q^{1/24}\prod_{n\text{ even}}
      (1-q^n)\prod_{n\text{ odd}}(1+q^n) \\
      &=e^{2\pi i/48}q^{1/24}\prod_{n=1}^\infty(1-q^{2n})
      \prod_{n=1}^\infty\frac{1+q^{n}}{1+q^{2n}} \\
      &=e^{2\pi i/48}q^{1/24}\prod_{n=1}^\infty(1-q^{2n})
      \prod_{n=1}^\infty\frac{(1-q^{2n})^2}{(1-q^n)(1-q^{4n})}\\
      &=e^{2\pi/48}\frac{\eta(2\tau)^3}{\eta(\tau)\eta(4\tau)}\\
      \eta(\tau+1/3)\eta(\tau+2/3)
      &=e^{2\pi i/24}q^{1/12}\prod_{3|n}(1-q^n)^2
      \prod_{3\nmid n}(1-e^{2\pi i/3}q^n)(1-q^{4\pi i/3}q^n)\\
      &=e^{2\pi i/24}q^{1/12}\prod_{n=1}^\infty(1-q^{3n})^2
      \prod_{3\nmid n}\frac{1-q^{3n}}{1-q^n} \\
      &=e^{2\pi i/24}q^{1/12}\prod_{n=1}^\infty(1-q^{3n})
      \prod_{n=1}^\infty\frac{(1-q^{3n})^2}{(1-q^{9n})(1-q^n)}\\
      &=e^{2\pi i/24}\frac{\eta(3\tau)^4}{\eta(\tau)\eta(9\tau)}
    \end{split}
  \end{equation}
  to check that the remaining $\eta_{m,k}$ can all be expressed as a
  product of those in \eqref{equation: product of eta mk}.
  For instance, we have
  \begin{equation} \label{equation: 3/4}
    \begin{split}
      \eta(\tau+3/4)&=\eta(\tau+1/4+1/2)
      =\epsilon\frac{\eta(2\tau+1/2)^3}{\eta(\tau+1/4)\eta(4\tau)}\\
      &=\frac\epsilon{\eta(\tau+1/4)\eta(4\tau)}\left(
        \frac{\eta(4\tau)^3}{\eta(2\tau)\eta(8\tau)}\right)^3
      =\frac{\epsilon\eta(4\tau)^8}
      {\eta(\tau+1/4)\eta(2\tau)^3\eta(8\tau)^3}.
    \end{split}
  \end{equation}
  Here $\epsilon$ represents some root of unity and may not be the
  same at each occurrence.
\end{Remark}

As an application of our determination of modular units, in the next
theorem, we prove that $\sC_\Q(N)=\sC(N)(\Q)$ for $N=n^2M$ with $n|24$
and $M$ squarefree.

\begin{Theorem} \label{theorem: equality} Assume that $N=n^2M$ with
  $n|24$ and $M$ squarefree. Then
  $$
  \sC_\Q(N)=\sC(N)(\Q).
  $$
\end{Theorem}


\begin{Remark}
In fact, our main motivation for undertaking this research project is
to seek for examples with $\sC(N)(\Q)\neq\sC_\Q(N)$. Such an example
will be a direct counterexample to the conjecture that
$\sC_\Q(N)=J_0(N)(\Q)_\tor$. However, after computing many examples
and studying properties of modular units more thoroughly, we
found that the equality actually holds for levels under
consideration. In view of Theorem \ref{theorem: equality}, it is
perhaps reasonable to conjecture that the two inclusions in
\eqref{equation: inclusions} are both equalities for all levels $N$.
\end{Remark}


\section{Modular units on $X_0(N)$}
We remind the reader that the level $N$ is assumed to be $n^2M$ with
$n|24$ and $M$ squarefree. We first recall the transformation formula
for the Dedekind eta function.

\begin{Lemma}[{\cite[Pages 125--127]{Weber}}] \label{lemma: eta}
For $\gamma=\SM abcd\in\SL(2,\Z)$, the transformation formula for
$\eta(\tau)$ is given by, for $c=0$, 
$$
  \eta(\tau+b)=e^{2\pi ib/24}\eta(\tau),
$$
and, for $c\neq 0$,
$$
  \eta(\gamma\tau)=\epsilon(a,b,c,d)\sqrt{\frac{c\tau+d}i}\eta(\tau)
$$
with
\begin{equation}
\label{equation: epsilon}
  \epsilon(a,b,c,d)=
  \begin{cases}\displaystyle
   \left(\frac dc\right)i^{(1-c)/2}
   e^{2\pi i\left(bd(1-c^2)+c(a+d)\right)/24},
    &\text{if }c\text{ is odd},\\
  \displaystyle 
  \left(\frac cd\right)e^{2\pi i\left(ac(1-d^2)+d(b-c+3)\right)/24},
    &\text{if }d\text{ is odd},
  \end{cases}
\end{equation}
where $\JS dc$ is the Jacobi symbol.
\end{Lemma}

\begin{Lemma} \label{lemma: eta mk}
  Assume that $m|N$. Let $h=h(m)$ and for an integer $k$, let
  $\eta_{m,k}(\tau):=\eta(m\tau+k/h)$.
  \begin{enumerate}
  \item[(a)] For $\gamma=\SM ab{c}d\in\Gamma_0(N)$,
  we have, when $c=0$
  $$
  \eta_{m,k}(\tau+b)=e^{2\pi ibm/24}\eta_{m,k}(\tau)
  $$
  and when $c\neq 0$,
  $$
  \eta_{m,k}(\gamma\tau)=\epsilon\left(a+\frac{kc}{hm},\frac{k(d-a)}h
        +bm-\frac{k^2c}{h^2m},\frac{c}m,d-\frac{kc}{hm}\right)
      \sqrt{\frac{c\tau+d}i}\eta_{m,k}(\tau),
  $$
  where $\epsilon$ is defined by \eqref{equation: epsilon}.
  \item[(b)] Let $a/c$ with $c|N$ be a cusp of $X_0(N)$. Write
    $(mha+kc)/hc$ in the reduced form $a'/c'$, $(a',c')=1$. Then the
    order of $\eta_{m,k}(\tau)$ at $a/c$ is
    $$
    \frac{cN}{24m(c')^2(c,N/c)}.
    $$
  \end{enumerate}
\end{Lemma}

Note that when $k=0$, we have $c'=c/(m,c)$ and the formula shows that
the order of $\eta(m\tau)$ at $a/c$ is
$$
\frac{N(m,c)^2}{24mc(c,N/c)},
$$
agreeing with the formula given in \cite[Proposition 3.2.8]{Ligozat}.

\begin{proof} It is clear that $\eta_{m,k}(\tau+b)=e^{2\pi
    ibm/24}\eta_{m,k}(\tau)$.
  Let $\sigma=\SM{mh}k0h$ so that $\eta_{m,k}(\tau)=\eta(\sigma\tau)$.  Let
  $$
  \gamma'=\sigma\gamma\sigma^{-1}
 =\M{a+kc/mh}{k(d-a)/h+bm-k^2c/mh^2}{c/m}{d-kc/mh}.
  $$
  Since $h$ is a divisor of $24$, we have $a\equiv d\mod h$ and
  $\gamma'\in\SL(2,\Z)$\footnote{This is where the assumption $h|24$
    is required. For 
    general $h$, $\eta_{m,k}|\SM abcd$ will equal to
    $\epsilon\eta_{m,k'}$ for some root of unity $\epsilon$, where
    $k'$ is an integer satisfying $ak'\equiv dk\mod h$.}.
  Then by Lemma \ref{lemma: eta}
  \begin{equation*}
    \begin{split}
      \eta_{m,k}(\gamma\tau)&=\eta(\sigma\gamma\tau)
      =\eta(\gamma'\sigma\tau) \\
      &=\epsilon\left(a+\frac{kc}{mh},\frac{k(d-a)}h
        +bm-\frac{k^2c}{mh^2},\frac{c}m,d-\frac{kc}{mh}\right)
      \sqrt{\frac{c\tau+d}i}\eta_{m,k}(\tau).
    \end{split}
  \end{equation*}
  We now prove Part (b).

  Let $a/c$ with $c|N$ be a cusp of $X_0(N)$. Let $b$, $d$, $b'$, and
  $d'$ be integers such that $\gamma=\SM
  abcd,\SM{a'}{b'}{c'}{d'}\in\SL(2,\Z)$. We check that
  $$
  m\M abcd\tau+\frac kh=\M{a'}{b'}{c'}{d'}
  \left(\frac{c(c\tau+d)}{m(c')^2}-\frac{d'}{c'}\right).
  $$
  It follows that
  $$
  \eta_{m,k}(\gamma\tau)=u\sqrt{\frac{c\tau+d}i}
  \eta\left(\frac{c(c\tau+d)}{m(c')^2}-\frac{d'}{c'}\right)
  $$
  for some nonzero complex number $u$. Since a cusp of level $c$ on
  $X_0(N)$ has width $N/c(c,N/c)$, we find that the order of
  $\eta_{m,k}(\tau)$ at $a/c$ is
  $$
  \frac{c^2}{24m(c')^2}\cdot\frac N{c(c,N/c)}
  =\frac{cN}{24m(c')^2(c,N/c)}.
  $$
  This completes the proof of the lemma.
\end{proof}

\begin{Corollary} \label{corollary: orders at infinity and 0}
  If the product in \eqref{equation: product of eta mk} is a
  modular function on $\Gamma_0(N)$, then the integers $e_{m,k}$
  satisfy
  \begin{equation} \label{equation: condition from weight}
    \sum_{m,k}e_{m,k}=0,
  \end{equation}
  \begin{equation} \label{equation: condition from infinity}
  \sum_{m,k}me_{m,k}\equiv 0\mod 24,
  \end{equation}
  and
  \begin{equation} \label{equation: condition from 0}
  \sum_{m,k}\frac{N(h(m),k)^2}{mh(m)^2}e_{m,k}\equiv 0\mod 24.
  \end{equation}
\end{Corollary}

\begin{proof} In order for the product to be a modular function on
  $\Gamma_0(N)$, it is necessary that its weight is $0$ and its orders
  at $\infty$ and $0$ are integers. The condition that the weight is
  $0$ translates to \eqref{equation: condition from weight}.
  Also, the order of $\eta_{m,k}$ at $\infty$ is $m/24$. Hence the
  condition that the order at $\infty$ is an integer translates to
  \eqref{equation: condition from infinity}. Finally, the order of
  $\eta_{m,k}$ at $0$ is determined by Part (b) of Lemma \ref{lemma:
    eta mk} (with $a=0$, $c=1$, $a'=k/(h(m),k)$, and
  $c'=h(m)/(h(m),k)$). We find that it is
  $$
  \frac{N}{24m(h(m)/(h(m),k))^2}.
  $$
  This explains the condition \eqref{equation: condition from 0}.
\end{proof}

\begin{Lemma} \label{lemma: simplify 1}
  Assume that
  $$
  f(\tau)=\prod_{m|N}\prod_{k=0}^{h(m)-1}\eta_{m,k}^{e_{m,k}}
  $$
  is a product satisfying the three conditions in Corollary
  \ref{corollary: orders at infinity and 0}. Let $G$ be the subgroup
  of $\Gamma_0(N)$ generated by $\SM1101$ and $\SM10N1$. Then for
  $\gamma\in\Gamma_0(N)$, the value of the root of unity
  $\mu$ in $f(\gamma\tau)=\mu f(\tau)$ depends only on the
  right coset $G\gamma$ of $\gamma$ in $\Gamma_0(N)$.
\end{Lemma}

\begin{proof} Let $\sigma=\SM1101$ and $\sigma'=\SM10N1$. The
  condition in \eqref{equation: condition from infinity}
  clearly implies that
  $$
  f(\sigma\gamma\tau)=f(\gamma\tau)
  $$
  for all $\gamma\in\Gamma_0(N)$. Likewise, since $\sigma'$ is a
  generator of the isotropy subgroup of the cusp $0$, the condition
  \eqref{equation: condition from 0} implies that
  $$
  f(\sigma'\gamma\tau)=f(\gamma\tau)
  $$
  for all $\gamma\in\Gamma_0(N)$. (More concretely, we may set
  $g(\tau)=f(-1/N\tau)$ and verify that $f(\sigma'\tau)=f(\tau)$ holds
  if and only if $g(\tau-1)=g(\tau)$ holds. Then notice that the
  latter follows from \eqref{equation: condition from 0}.) This proves
  the lemma.
\end{proof}

\begin{Lemma} \label{lemma: simplify 2}
  Let $G$ be the subgroup of $\Gamma_0(N)$ generated by
  $\sigma=\SM1101$ and $\sigma'=\SM10N1$. Then every right coset in
  $G\backslash\Gamma_0(N)$ contains an element $\SM abcd$ such that
  $24N|c$.
\end{Lemma}

\begin{proof} Assume that $\gamma=\SM abcd\in\Gamma_0(N)$. Let
  $s=(a,6)$ and 
  $t=6/s$. Since $(a,c)=1$, we have $(c,s)=1$ and hence $(a+tc,6)=1$.
  Let $c'=c/N$ and $r$ be an integer such that $r(a+tc)+c'\equiv
  24$. Now
  $$
  \M10{rN}1\M1t01\M abcd=\M{a+tc}{b+td}{c+Nr(a+tc)}{d+rN(b+td)}.
  $$
  By our choice of $r$, the $(2,1)$-entry of the last matrix is
  divisible by $24N$. This proves the
  lemma\footnote{This proof is suggested by one of the referees.}.
\end{proof}

%
%

We are now ready to prove Theorem \ref{theorem: modular
  units}.

\begin{proof}[Proof of Theorem \ref{theorem: modular
    units}]
  Let
  $$
  f(\tau)=\prod_{m|N}\prod_{k=0}^{h(m)-1}
  \eta_{m,k}^{e_{m,k}}.
  $$
  By Corollary \ref{corollary: orders at infinity and 0}, in order for
  $f$ to be a modular function on $\Gamma_0(N)$, it is necessary that
  $f$ satisfies the three conditions \eqref{equation: condition from
    weight}, \eqref{equation: condition from infinity}, and
  \eqref{equation: condition from 0}, which we assume from now on.

  Let $\gamma=\SM abcd\in\Gamma_0(N)$. By Lemmas \ref{lemma: simplify
    1} and \ref{lemma: simplify 2}, we may assume that
  $24N|c$. When $c=0$, i.e., when $\gamma=\SM1b01$, we have
  $$
  \eta_{m,k}(\tau+b)=e^{2\pi ibm/24}\eta_{m,k}(\tau)
  $$
  and hence $f(\tau+b)=f(\tau)$ by Condition \eqref{equation:
    condition from infinity} that $\sum_{m,k}me_{m,k}\equiv
  0\mod 24$.
  
  When $c\neq 0$, we apply Lemma \ref{lemma: eta mk} and obtain
  $$
  \eta_{m,k}(\gamma\tau)=\epsilon\left(
    a+\frac{kc}{hm},bm+\frac{k(d-a)}h-\frac{k^2c}{h^2m},
    \frac cm,d-\frac{kc}{hm}\right)\sqrt{\frac{c\tau+d}i}
  \eta_{m,k}(\tau),
  $$
  where $\epsilon$ is given by \eqref{equation: epsilon}. Since
  $24N|c$, we have
  $$
  24\Big|\frac{kc}{hm},\ \frac{k^2c}{h^2m},\ \frac cm
  $$
  for all $m|N$ and all $k$. Hence,
  \begin{equation*}
    \begin{split}
    &\epsilon\left(
    a+\frac{kc}{hm},bm+\frac{k(d-a)}h-\frac{k^2c}{h^2m},
    \frac cm,d-\frac{kc}{hm}\right)
  =\JS{c/m}{d-kc/hm}e^{2\pi iS/24},
    \end{split}
  \end{equation*}
  where
  $$
  S=d\left(bm+\frac{k(d-a)}h+3\right).
  $$
  Now $h$ is relatively prime to $d-kc/hm$ since $h^2m|c$ and
  $(d,N)=1$. Therefore,
  $$
  \JS{c/m}{d-kc/hm}=\JS{c/h^2m}{d-kc/hm}=\JS{c/h^2m}d
  =\JS{cm}d.
  $$
  It follows that, by \eqref{equation: condition from weight},
  $f(\gamma\tau)=\mu_1\mu_2f(\tau)$, where
  $$
  \mu_1=\prod_{m,k}\JS md^{e_{m,k}}, \qquad
  \mu_2=\exp\Bigg\{\frac{2\pi i}{24}\sum_{m,k}
    e_{m,k}\left(bdm+\frac{kd(d-a)}h\right)\Bigg\}.
  $$
  Since $f$ is assumed to satisfy \eqref{equation: condition from
    infinity}, we have
  $$
  \sum_{m,k}e_{m,k}bdm\equiv 0\mod 24.
  $$
  Also, because $24|c$, we have $24|(d-a)$, say, $d-a=24d'$. We deduce
  that
  $$
  \mu_2=\exp\Bigg\{2\pi i\frac{dd'}n\sum_{m,k}e_{m,k}\frac{kn}h\Bigg\}.
  $$
  Since the value of $\mu_1$ can only be $\pm1$ and that of $\mu_2$ is
  an $n$th root of unity, when $n=3$, we need to
  have $\mu_1=\mu_2=1$. By varying $\gamma=\SM abcd$, we conclude that
  in the case $n=3$, $f(\gamma\tau)=f(\tau)$ holds for all
  $\gamma\in\Gamma_0(N)$ if and
  only if $\prod_{m,k}m^{e_{m,k}}$
  is the square of a rational number and
  $$
  \sum_{m,k}e_{m,k}\frac{kn}h\equiv 0\mod 3,
  $$
  which is the equivalent to $\sum_{m,k}e_{m,k}k\equiv0\mod3$.

  In the case $n$ is even, we need $\mu_1=\mu_2=1$ or
  $\mu_1=\mu_2=-1$. As $d$ is relatively prime to $n$ and there are
  $\SM abcd\in\Gamma_0(N)$ such that $24|c$ and $d'=(d-a)/24$ is also
  relatively prime to $n$, we find that
  the sum $\sum_{m,k}e_{m,k}kn/h$ must be a multiple of $n/2$, say,
  $$
  \sum_{m,k}e_{m,k}\frac{kn}h=\frac{k'n}2
  $$
  for some $k'\in\Z$. Then
  $$
  \mu_2=(-1)^{d'k'}.
  $$
  Now since $32|c$, we have
  $$
  (-1)^{d'}=(-1)^{(d-a)/8}=(-1)^{(d^2-1)/8}=\JS 2d.
  $$
  It follows that
  $$
  \mu_1\mu_2=\JS 2d^{k'}\prod_{m,k}\JS md^{e_{m,k}}.
  $$
  By varying $\SM abcd\in\Gamma_0(N)$, we see that $f$ is a modular
  function on $\Gamma_0(N)$ if and only if
  $$
  2^{k'}\prod_{m,k}m^{e_{m,k}}
  $$
  is the square of a rational number, in addition to the three
  conditions \eqref{equation: condition from weight}, \eqref{equation:
    condition from infinity}, and \eqref{equation: condition from 0},
  or equivalently, the odd part of $\prod_{m,k}m^{e_{m,k}}$ is the
  square of a rational number and
  $$
  \sum_{m,k}e_{m,k}\left(\frac{kn}h+\frac n2\mathrm{ord}_2(m)
    \right)\equiv 0\mod n.
  $$
  This completes the proof of Theorem \ref{theorem: modular units}.
\end{proof}

We next prove Theorem \ref{theorem: modular units 2}.

\begin{proof}[Proof of Theorem \ref{theorem: modular units 2}]
  Theorem \ref{theorem: modular units 2} will follow once we prove
  the following four claims.
  \begin{enumerate}
    \item[(a)] The number of pairs $(m,k)$ with $m|N$ and $0\le
      k\le\phi(h(m))-1$ is equal to the number of cusps of $X_0(N)$.
    \item[(b)] There are no multiplicative relations among
      $\eta_{m,k}$ with $m|N$ and $0\le k\le\phi(h(m))-1$.
    \item[(c)] Let $\sU_0$ be the subgroup of $\sU(N)$ formed by
      products $\prod_{m|N}\prod_{k=0}^{\phi(h(m))-1}\eta_{m,k}^{e_{m,k}}$
      satisfying the conditions in Theorem \ref{theorem: modular
        units}. Then $\sU_0$ is of finite index in $\sU(N)$, which
      implies that if $g\in\sU(N)$, then there exists a positive
      integer $\ell$ such that $cg^\ell$ is in $\sU_0$ for some
      nonzero complex number $c$.
    \item[(d)] If $g\in\sU(N)$ and $\ell$ is a positive integer such
      that $cg^\ell$ is in $\sU_0$ for some $c\in\C^\times$, then
      $c'g\in\sU_0$ for some $c'\in\C^\times$.
    \end{enumerate}
  To prove Claim (a), we first observe that for a given divisor $h_0$
  of $n$, a divisor $m$ of $N$ satisfies $h(m)=h_0$ if and only if
  $m|N/h_0^2$ and $N/mh_0^2$ is squarefree. Let $\mu$ be the Mobius
  function so that $\mu^2$ is the characteristic function of
  squarefree integers. Then the number of pairs $(m,k)$ with $m|N$ and
  $0\le k\le\phi(h(m))-1$ is
  $$
  \sum_{h|n}\phi(h)\sum_{m|N/h^2}\mu(m)^2.
  $$
  On the other hand, the number of cusps of $X_0(N)$ is
  $$
  \sum_{m|N}\phi((m,N/m))=\sum_{h|n}\phi(h)
  \sum_{m'|N/h^2,(m',N/m'h^2)=1}1
  =\sum_{h|n}\phi(h)2^{\omega(N/h^2)},
  $$
  where for a positive integer $k$, $\omega(k)$ denotes the number of
  prime factors of $k$. Now we check that both functions
  $k\mapsto\sum_{m|k}\mu(m)^2$ and $k\mapsto 2^{\omega(k)}$ are
  multiplicative and agree on prime powers. Therefore, we have
  $$
  \sum_{m|N/h^2}\mu(m)^2=2^{\omega(N/h^2)}.
  $$
  This proves Claim (a).

  We next prove Claim (b). 
  Assume that $e_{m,k}$ are integers such that
  $$
  \prod_{m|N}\prod_{k=0}^{\phi(h(m))-1}\eta_{m,k}^{e_{m,k}}
  $$
  is a constant function. Considering the second term in its Fourier
  expansion, we find that
  $$
  \sum_{k=0}^{\phi(n)-1}e_{1,k}\zeta_n^k=0, \quad
  \zeta_n=e^{2\pi i/n}.
  $$
  Recall that $1,\ldots,\zeta_n^{\phi(n)-1}$ form a basis of
  $\Q(\zeta_n)$ over $\Q$ (see, for instance, \cite[Theorem
  6.4]{Milne}). Hence $e_{1,k}=0$ for all
  $k=0,\ldots,\phi(n)-1$. Similarly, 
  by considering the Fourier coefficients of $q^m$ for the next
  divisor $m$ of $N$, we find that $e_{m,k}=0$ for
  $k=0,\ldots,\phi(h(m))-1$ for the next divisor $m$ of $N$.
  Continuing in this way, we find that $e_{m,k}=0$ for all $(m,k)$.
  
  For Claim (c), we observe that $\sU_0$ contains at least those
  products having $24|e_{m,k}$ for all $e_{m,k}$ and $\sum e_{m,k}=0$.
  It follows that, by Claim (b), the rank of $\sU_0$ is at least
  $$
  \#\{(m,k):m|N,~0\le k\le \phi(h(m))-1\}-1,
  $$
  which, by Claim (a), is equal to the number of cusps of $X_0(N)$
  minus $1$. Therefore, $\sU_0$ and $\sU(N)$ have the same
  rank. We now prove Claim (d). 

  Let $g$ be a modular unit on $X_0(N)$. Without loss of generality,
  we may assume that the leading coefficient of $g$ is $1$. Since $g$
  is naturally also a modular unit on $X(N)$, by
  \cite{Kubert,Kubert-Lang-IV}, $g$ is a product of Siegel
  functions and, in the case $N$ is even, also functions of the forms
  $q^{-d/48}\prod_n(1+q^{d(n+1/2)})$ and $q^{d/24}\prod_n(1+q^{dn})$.
  (We refer the reader to \cite{Kubert-Lang-book} for the definition
  of Siegel functions.)
  Hence all its Fourier coefficients are algebraic integers. Also,
  by Claim (c), there exists a positive integer $\ell$ such that
  $g^\ell\in\sU_0$ up to a scalar, say,
  \begin{equation} \label{equation: g}
  g^\ell=c\prod_{m|N}\prod_{k=0}^{\phi(h(m))-1}\eta_{m,k}^{e_{m,k}}.
  \end{equation}

  Now for convenience, for a Puiseux series $f$
  in $q$, we let
  $$
  S(f)=\frac{\text{second nonzero term of }f}
  {\text{leading term of }f}.
  $$
  Comparing the $\ell$th roots of the two sides of \eqref{equation:
    g}, we find that
  $$
  S(g)=-\left(\sum_{k=0}^{\phi(n)-1}\frac{e_{1,k}}\ell\zeta_n^k\right)
  q, \quad \zeta_n=e^{2\pi i/n}.
  $$
  Since $S(g)$ is an algebraic integer and
  $1,\zeta_n,\ldots,\zeta_n^{\phi(n)-1}$ form an integral basis for
  the ring of integers in $\Q(\zeta_n)$, we must have
  $e_{1,k}/\ell\in\Z$ for all $k$. By the same token, by considering
  $S(g\prod_k\eta_{1,k}^{-e_{1,k}/\ell})$, we deduce that
  $e_{m,k}/\ell\in\Z$ for all $k$ for the next divisor $m$ of $N$.
  Continuing this way, we conclude that $g\in\sU_0$ up to a scalar.
  This completes the proof of Theorem \ref{theorem: modular units 2}.
\end{proof}

\section{The two groups $\sC(N)(\Q)$ and $\sC_\Q(N)$}
We will prove Theorem \ref{theorem: equality} in this section.
Let $D$ be a cuspidal divisor of degree $0$ on $X_0(N)$ such that
$D^\sigma\sim D$ for all $\sigma\in\Gal(\overline\Q/\Q)$. (Here $\sim$
denotes the linear equivalence between divisors.) Our goal is to
construct a $\Q$-rational cuspidal divisor $D'$ such that $D\sim D'$.

Assume that $[D]$ has order $r$ in $J_0(N)$ and $f$ is a modular unit
such that $\div f=rD$. By Theorem \ref{theorem: modular units 2}, we have
\begin{equation} \label{equation: f}
f=\prod_{m|N}\prod_{k=0}^{\phi(h(m))-1}\eta_{m,k}^{e_{m,k}}
\end{equation}
for some integers $e_{m,k}$ satisfying the four conditions in Theorem
\ref{theorem: modular units}. We first describe how $\Gal(\Q(e^{2\pi
  i/n})/\Q)$ acts on $\div f$.

\begin{Lemma} For an integer $\ell$ with $(\ell,N)=1$, let
  $\sigma_\ell$ be the element of $\Gal(\Q(e^{2\pi i/n})/\Q)$ that maps
  $e^{2\pi i/n}$ to $e^{2\pi i\ell/n}$. We have
  $$
  (\div\eta_{m,k})^{\sigma_\ell}=\div\eta_{m,\ell k}.
  $$
\end{Lemma}

\begin{proof} Let $\ell'$ be an integer such that
  $\ell\ell'\equiv1\mod N$. We first remark that because $n|24$, we
  have $\ell\equiv\ell'\mod n$. Thus, $\div\eta_{m,\ell
    k}=\div\eta_{m,\ell'k}$ since the two functions differ only by a
  root of unity. We will prove the lemma in the form
  $$
  (\div\eta_{m,k})^{\sigma_\ell}=\div\eta_{m,\ell'k}.
  $$

  Let $a/c$ with $c|N$ be a cusp of $X_0(N)$. Recall that
  the action of $\sigma_\ell$ maps the cusp $a/c$ to the cusp
  $a/\ell'c$ (see,
  for instance, \cite[Theorem 1.3.1]{Stevens-book}). This cusp
  $a/\ell'c$ is equivalent to $\ell'a/c$ (see, for example,
  \cite[Proposition 2.2.3]{Cremona}). Thus, to prove the lemma, it
  suffices to show that the order of $\eta_{m,k}$ at $a/c$ is equal to
  that of $\eta_{m,\ell'k}$ at $\ell'a/c$. Now by Lemma \ref{lemma:
    eta mk}, the former is
  $$
  \frac{cN}{24m(c')^2(c,N/c)},
  $$
  while the latter is
  $$
  \frac{cN}{24m(c'')^2(c,N/c)},
  $$
  where $c'$ and $c''$ are the denominators in the reduced froms of
  $(mha+kc)/hc$ and $(mh\ell'a+\ell'kc)/hc$, respectively. Since
  $\ell'$ is relatively prime to $hc$, we have $c'=c''$. Then the
  lemma follows.
\end{proof}

In view of the lemma, we naturally define
$$
f^{\sigma_\ell}:=\prod_{m|N}\prod_{k=0}^{\phi(h(m))-1}
\eta_{m,\ell k}^{e_{m,k}}
$$
for $\ell$ with $(\ell,N)=1$ so that
$$
\div f^{\sigma_\ell}=(\div f)^{\sigma_\ell}=rD^{\sigma_\ell}.
$$
Our strategy of proving the theorem is as follows.
\begin{enumerate}
  \item[(a)] We first show (case by case) that $r|e_{m,k}$ for all
    $(m,k)$ with $0<k<\phi(h(m))$. This is achieved by using the
    assumption that $D^\sigma\sim D$, $\sigma\in\Gal(\Q(e^{2\pi
      i/n})/\Q)$, which implies that $f^\sigma/f$
    is the $r$th power of some modular unit on $X_0(N)$.
  \item[(b)] Set $e_{m,k}'=e_{m,k}/r$ for $(m,k)$ with
    $0<k<\phi(h(m))$. Using the assumption that $(f^\sigma/f)^{1/r}$ is a
    modular unit again, we deduce some congruence relations for
    $e_{m,k}'$ from Theorem \ref{theorem: modular units}.
  \item[(c)] For each $(m,k)$ with $0<k<\phi(h(m))$, construct a
    function $\wt\eta_{m,k}$ such that
    \begin{enumerate}
    \item[(i)] $\wt\eta_{m,k}$ is a product of $\eta(d\tau)$, $d|N$,
    \item[(ii)] the ratio $\eta_{m,k}/\wt\eta_{m,k}$ satisfies Conditions
      (a), (b), and (c) in Theorem \ref{theorem: modular units}.
    \end{enumerate}
  \item[(d)] Define
    $$
    g=\prod_{(m,k):0<k<\phi(h(m))}\left(\frac{\eta_{m,k}}
      {\wt\eta_{m,k}}\right)^{e_{m,k}'}.
    $$
    Show that $g$ satisfies Condition (d) of Theorem \ref{theorem:
      modular units} using the congruence relations among $e_{m,k}'$
    from Step (b) so that $g$ is a modular unit. (In some cases, we
    may need to modify $g$ a little bit.)
  \item[(e)] Let $D'=D-\div g$, which is equivalent to $D$. Now we have
    $$
    rD'=rD-r\div g=\div(f/g^r).
    $$
    By our construction of $g$, we find that $f/g^r$ is a product of
    $\eta(d\tau)$, $d|N$, and hence has a $\Q$-rational divisor. This
    proves the theorem.
\end{enumerate}

To construct $\wt\eta_{m,k}$, we shall use the following lemma. 

\begin{Lemma} \label{lemma: wt eta}
  Assume that $m$, $h$, and $k$ are positive integers such that
  $h^2m|N$ and $(k,h)=1$. Then the orders of the functions $s(\tau)$
  defined below at the cusps $\infty$ and $0$ are both integers.
  \begin{enumerate}
  \item[(a)] Assume that $3|h$. Set $h'=h/3$ and let
     $$
    s(\tau)=\frac{\eta(m\tau+k/h)\eta(3m\tau+k/h')^4}
    {\eta(m\tau+k/h')^4\eta(9m\tau+k/h')}.
    $$
  \item[(b)] Assume that $4|h$. Set $h'=h/4$ and let
    $$
    s(\tau)=\frac{\eta(m\tau+k/h)\eta(m\tau+k/h')\eta(4m\tau+k/h')^3
      \eta(16m\tau+k/h')}{\eta(2m\tau+k/h')^3\eta(8m\tau+k/h')^3}.
    $$
  \end{enumerate}
\end{Lemma}

\begin{proof}
  The order at $\infty$ is clearly $0$. By Lemma \ref{lemma: eta mk},
  the order of the function in Part (a) at $0$ is
  $$
  \frac1{24}\left(\frac N{mh^2}+\frac{4N}{3m(h')^2}
    -\frac{4N}{m(h')^2}-\frac N{9m(h')^2}\right)
  =\frac N{24mh^2}(1+12-36-1)=-\frac N{mh^2},
  $$
  while that of the function in Part (b) is
  \begin{equation*}
    \begin{split}
     &\frac1{24}\left(\frac N{mh^2}+\frac N{m(h')^2}+\frac{3N}{4m(h')^2}
       +\frac N{16m(h')^2}-\frac{3N}{2m(h')^2}
       -\frac{3N}{8m(h')^2}\right) \\
       &\qquad=\frac N{24mh^2}(1+16+12+1-24-6)=0.
    \end{split}
  \end{equation*}
  The orders are indeed integers.
\end{proof}

We now describe our construction of $g$ case by case. For convenience,
all equalities among modular units stated below hold only up to
nonzero scalars. 
Note that the cases $n=1$ and $n=2$ are trivial since every cusp is
$\Q$-rational in these cases.

\subsection{Case $n=3$}
  Let $D$, $f$, and $\eta_{m,k}$ be given as above and
  $\sigma=\sigma_{-1}$ be the nontrivial element in $\Gal(\Q(e^{2\pi
    i/3})/\Q)$. As explained in the description of our strategy,
  we know that $f^\sigma/f$ is the $r$th power of a modular unit.
  Now we have
  $$
  f^\sigma/f=\prod_{m|M}\left(\frac{\eta(m\tau-1/3)}
    {\eta(m\tau+1/3)}\right)^{e_{m,1}}.
  $$
  (Note that $k=1$ occurs only when $m|M$.)
  Using \eqref{equation: eta relations}, we may write it as
  $$
  f^\sigma/f=\prod_{m|M}\left(
  \frac{\eta(3m\tau)^4}{\eta(m\tau+1/3)^2\eta(m\tau)\eta(9m\tau)}
  \right)^{e_{m,1}}.
  $$
  Since this is the $r$th power of some modular unit, by the uniqueness
  of product expression described in Theorem 
  \ref{theorem: modular units 2}, we must have $r|e_{m,k}$ for all
  $(m,k)$ with $k=1$. (Alternatively, we may follow the argument for
  Claim (c) in the proof of Theorem \ref{theorem: modular units 2} to
  show that $r|e_{m,1}$ for all $m$.) Set $e_{m,1}'=e_{m,1}/r$.
  Note that since $(f^\sigma/f)^{1/r}$ is a modular unit, we have
  \begin{equation} \label{equation: emk' mod 3}
  \sum_{m|M}e_{m,1}'\equiv 0\mod 3,
  \end{equation}
  by Theorem \ref{theorem: modular units 2}.

  For $m|M$, define
  \begin{equation*}
  \wt\eta_{m,1}(\tau)=\frac{\eta(m\tau)^4\eta(9m\tau)}
  {\eta(3m\tau)^4}
  \end{equation*}
  and set
  $$
  g(\tau)=\prod_{m|M}\left(\frac{\eta_{m,1}}
    {\widetilde\eta_{m,1}}\right)^{e_{m,1}'}.
  $$
  By Lemma \ref{lemma: wt eta}, the order of $g$ at $\infty$ and $0$
  are integers and hence $g$ satisfies Conditions (a), (b), and (c) in
  Theorem \ref{theorem: modular units}. Also, by \eqref{equation: emk'
    mod 3}, Condition (d) is fulfilled. Hence $g$ is a modular unit on
  $X_0(N)$, and $D'=D-\div g$ is a $\Q$-rational cuspidal divisor
  equivalent to $D$.

\subsection{Case $n=6$} Let $\sigma$ be the nontrivial element in
$\Gal(\Q(e^{2\pi i/3})/\Q)$. As in the case $n=3$, we can show that
$r|e_{m,1}$ for all $m|N$ with $\phi(h(m))\neq 1$ (i.e., $h(m)=3$ or
$h(m)=6$). Set $e_{m,1}'=e_{m,1}/r$ for such $m$. The fact that
$(f^\sigma/f)^{1/r}$ is a modular unit implies that
$$
\prod_{m|N,h(m)=3,6}\left(\frac{\eta_{m,-1}}{\eta_{m,1}}\right)^{e_{m,1}'}
$$
is a modular unit, which in turn shows that
\begin{equation} \label{equation: emk' mod 6}
-2\sum_{m|N,h(m)=6}e_{m,1}'+2\sum_{m|N,h(m)=3}e_{m,1}'\equiv 0\mod 6,
\end{equation}
by Condition (d) of Theorem \ref{theorem: modular units}.

For $m|N$ with $h(m)=6$, define
$$
\wt\eta_{m,1}(\tau)=\frac{\eta(m\tau+1/2)^4\eta(9m\tau+1/2)}
{\eta(3m\tau+1/2)^4},
$$
and for $m|N$ with $h(m)=3$, define
$$
\wt\eta_{m,1}(\tau)=\frac{\eta(m\tau)^4\eta(9m\tau)}
{\eta(3m\tau)^4}.
$$
Note that by \eqref{equation: eta relations}, $\eta(m\tau+1/2)$
can be written as a product of $\eta(d\tau)$, $d|N$. Set
$$
g=\prod_{m|M,h(m)=3,6}
\left(\frac{\eta_{m,1}}{\wt\eta_{m,1}}\right)^{e_{m,1}'}
$$
By Lemma \ref{lemma: eta mk}, $g$ satisfies Conditions (a), (b), and
(c) of Theorem \ref{theorem: modular units}. Also, the left-hand side
of \eqref{equation: mod 2} for the function $g$ is
$$
-2\sum_{m|N:h(m)=6}e_{m,1}'+2\sum_{m|N:h(m)=3}e_{m,1}',
$$
which by \eqref{equation: emk' mod 6}, is congruent to $0$ modulo $6$.
Hence Condition (d) is also satisfied, and $g$ is a modular unit. This
proves the theorem for the case $n=6$.

\subsection{Case $n=4$} Let $\sigma$ be the nontrivial element of
$\mathrm{Gal}(\Q(\sqrt{-1})/\Q)$. Again, we omit the proof of
$r|e_{m,1}$ for all $m|M$. (Note that $h(m)=4$ if and only if
$m|M$.) Set $e_{m,1}'=e_{m,1}/r$ for
those $m$. The fact that
$$
\prod_{m|M}\left(\frac{\eta(m\tau-1/4)}{\eta(m\tau+1/4)}
\right)^{e_{m,1}'}
=\left(\frac{f^\sigma}f\right)^{1/r}
$$
is a modular unit implies that
$$
\sum_{m|M}e_{m,1}'\equiv 0\mod 2,
$$
by Condition (d) in Theorem \ref{theorem: modular units}.

Define
$$
\widetilde\eta_{m,1}(\tau)=\frac
{\eta(2m\tau)^3\eta(8m\tau)^3}{\eta(m\tau)\eta(4m\tau)^3
  \eta(16m\tau)},
$$
and
$$
g=\prod_{m|M}\left(\frac{\eta_{m,1}}{\widetilde\eta_{m,1}}
\right)^{e_{m,1}'}.
$$
By Lemma \ref{lemma: eta mk}, $g$ satisfies Conditions (a),
(b), (c) in Theorem \ref{theorem: modular units}. Moreover,
if
$$
\sum_{m|M}e_{m,1}'\equiv 0\mod 4,
$$
then Condition (d) is also fulfilled (note that for an individual
$\eta_{m,1}/\wt\eta_{m,1}$, the sum of the left-hand side of
\eqref{equation: mod 2} is congruent to $1$ modulo $4$) and hence $g$ is a modular unit
on $X_0(N)$. If
$$
\sum_{m|M}e_{m,1}'\equiv 2\mod 4
$$
instead, we replace $g$ by
$$
g=\frac{\eta(\tau)^2\eta(4\tau)^7}{\eta(2\tau)^7\eta(8\tau)^2}
\prod_{m|M}\left(\frac{\eta_{m,1}}{\widetilde\eta_{m,1}}
\right)^{e_{m,1}'},
$$
which is a modular unit under the assumption $\sum_{m|M}e_{m,1}'\equiv
2\mod 4$. Either way, we find that $D'=D-\div g$ is a $\Q$-rational
cuspidal divisor linearly equivalent to $D$. 

\subsection{Case $n=8$}
For $a\in\{\pm1,\pm3\}$, let $\sigma_a$ be the element of
$G=\Gal(\Q(e^{2\pi i/8})/\Q)$ that maps $e^{2\pi i/8}$ to $e^{2\pi
  ia/8}$.
We have
$$
f^{\sigma_a}/f=\prod_{m|N,h(m)=8}\prod_{k=1}^3
\left(\frac{\eta_{m,ak}}{\eta_{m,k}}\right)^{e_{m,k}}
\times\prod_{m|N,h(m)=4}\left(
  \frac{\eta_{m,a}}{\eta_{m,1}}\right)^{e_{m,1}}
$$
for $a\in\{\pm1,\pm3\}$. As $f^{\sigma_a}/f$ is the $r$th power of some
modular unit, by considering $a=3$ and using \eqref{equation: eta
  relations}, we find that $r|e_{m,2}$ for $m$ with $h(m)=8$ and
$r|e_{m,1}$ for $m$ with $h(m)=4$. By considering $a=-3$ instead, we
conclude also that $r|e_{m,1},e_{m,3}$ for $m$ with $h(m)=8$. Set
$e_{m,k}'=e_{m,k}/r$ for those $(m,k)$. The fact that
$$
\prod_{m|N,h(m)=8}\prod_{k=1}^3
\left(\frac{\eta_{m,-k}}{\eta_{m,k}}\right)^{e_{m,k}'}
\times\prod_{m|N,h(m)=4}\left(
  \frac{\eta_{m,-1}}{\eta_{m,1}}\right)^{e_{m,1}'}
=\left(\frac{f^{\sigma_{-1}}}f\right)^{1/r}
$$
is a modular unit implies that
\begin{equation} \label{equation: n=8}
  \sum_{m|N,h(m)=8}(6e_{m,1}'+4e_{m,2}'+2e_{m,3}')+4
  \sum_{m|N,h(m)=4}e_{m,1}'\equiv0\mod 8,
\end{equation}
by Condition (d) of Theorem \ref{theorem: modular units}. Define
$\wt\eta_{m,k}$ by
$$
\wt\eta_{m,k}=
  \frac{\eta(2m\tau+1/2)^3\eta(8m\tau+1/2)^3}
  {\eta(m\tau+1/2)\eta(4m\tau+1/2)^3\eta(16m\tau+1/2)}
$$
for $(m,k)$ with $h(m)=8$ and $k=1,3$, and by
$$
\wt\eta_{m,k}=
  \frac{\eta(2m\tau)^3\eta(8m\tau)^3}
  {\eta(m\tau)\eta(4m\tau)^3\eta(16m\tau)}
$$
for $(m,k)$ with $(h(m),k)=(4,1)$ or $(h(m),k)=(8,2)$.
Then set
$$
g_0=\prod_{m|N,h(m)=8}\prod_{k=1}^3\left(\frac{\eta_{m,k}}
  {\wt\eta_{m,k}}\right)^{e_{m,k}'}\times
\prod_{m|N,h(m)=4}\left(\frac{\eta_{m,1}}{\wt\eta_{m,1}}
\right)^{e_{m,1}'}.
$$
By Lemma \ref{lemma: wt eta}, this function $g_0$ satisfies Conditions
(a), (b), and (c) in Theorem
\ref{theorem: modular units}. Furthermore, the left-hand side of
\eqref{equation: mod 2} for $g_0$ is congruent to
$$
\sum_{m|N,h(m)=8}(5e_{m,1}'+6e_{m,2}'+7e_{m,3})
+2\sum_{m|N,h(m)=4}'e_{m,1}'
$$
modulo $8$. By \eqref{equation: n=8}, this sum is a
multiple of $4$. We set
$$
g=\begin{cases}
  g_0, &\text{if the sum is divisible by }8, \\
  g_0\eta(\tau)^2\eta(4\tau)^7/\eta(2\tau)^7\eta(8\tau)^2,
  &\text{if the sum is congruent to }4\text{ modulo }8.
  \end{cases}
$$
Then this is a modular unit on $X_0(N)$, as we are required to construct.

\subsection{Case $n=12$ or $n=24$} The idea of proof is similar to
previous cases, so we will only sketch the argument. Let $n=12$ or
$n=24$. By using the property that $f^\sigma/f$ is the $r$th power of
a modular unit for all $\sigma\in\Gal(\Q(e^{2\pi i/n})/\Q)$, we can
show $r|e_{m,k}$ for all $(m,k)$ with $k\neq0$. Set
$e_{m,k}'=e_{m,k}/r$ for those $(m,k)$. Then the fact that
$$
\left(\frac{f^{\sigma_{-1}}}f\right)^{1/r}
=\prod_{m|N}\prod_{k=1}^{\phi(h(m))-1}\left(\frac{\eta_{m,-k}}
  {\eta_{m,k}}\right)^{e_{m,k}'}
$$
is a modular unit implies that
\begin{equation} \label{equation: n=12, 24}
2\sum_{h|n}\frac nh\sum_{(m,k):h(m)=h,k\neq 0}ke_{m,k}'\equiv 0\mod n
\end{equation}
by Condition (d) of Theorem \ref{theorem: modular units}. Now Lemma
\ref{lemma: wt eta} provides two procedures to find a function $\iota$
such that $\iota$ is a product of $\eta(dm\tau+k/(h/3))$ or a
product of $\eta(dm\tau+k/(h/4))$, depending on whether $h$ is
divisible by $3$ or $4$, and $\eta_{m,k}/\iota$ has weight
$0$, and its order at the cusps $\infty$ and $0$ are integers.
Using these two procedures or the combination of the two procedures,
we can construct $\wt\eta_{m,k}$ that is a product of $\eta(d\tau)$,
$d|N$, or a product of $\eta(d\tau+1/2)$, $d|N/4$, in the case
$h(m)=6,8,24$ such that it has
weight $0$ and its order at $\infty$ and $0$ are integers. For instance, for $\eta_{1,1}=\eta(\tau+1/24)$, we apply
Part (a) of Lemma \ref{lemma: wt eta} and find that the function
$\iota$ can be chosen to be
$$
\frac{\eta(\tau+1/24)\eta(3\tau+1/8)^4}
{\eta(\tau+1/8)^4\eta(9\tau+1/8)}.
$$
Then for each $\eta(d\tau+1/8)$, $d|9$, we apply Part (b) of the same
lemma and find that
$$
\frac{\eta(d\tau+1/8)\eta(d\tau+1/2)\eta(4d\tau+1/2)^3
  \eta(16d\tau+1/2)}{\eta(2d\tau+1/2)^3\eta(8d\tau+1/2)^3}
$$
has weight $0$ and integer orders at $\infty$ and $0$. From these two
procedures, we obtain a function $\wt\eta_{1,1}$ of the form
$\prod_{d|144}\eta(d\tau+1/2)$ such that $\eta_{1,1}/\wt\eta_{1,1}$
satisfies Conditions (a), (b), and (c) of Theorem \ref{theorem:
  modular units}.

Let $\wt\eta_{m,k}$ be the eta-products constructed above and consider
$$
g_0=\prod_{m|N}\prod_{k=1}^{\phi(h(m))-1}\left(
  \frac{\eta_{m,k}}{\wt\eta_{m,k}}\right)^{e_{m,k}'}.
$$
By construction, $g_0$ satisfies Conditions (a), (b), and (c) in
Theorem \ref{theorem: modular units}. Furthermore, the left-hand side
of \eqref{equation: mod 2} for $g_0$ is
$$
\sum_{h|n}\frac nh\sum_{(m,k):h(m)=h,k\neq 0}ke_{m,k}'
-\frac n2\sum_{h=6,8,24}\sum_{(m,k):h(m)=h,k\neq 0}e_{m,k}'
$$
modulo $n$. (Note that the second sum is coming from $\wt\eta_{m,k}$
that are products of the form $\eta(dm\tau+1/2)$ in the case
$h(m)=6,8,24$.) By \eqref{equation: n=12, 24}, the sum above is
congruent to $0$ modulo $n/2$. Set
$$
g=\begin{cases}
  g_0, &\text{if the sum is divisible by }n, \\
  g_0\eta(\tau)^2\eta(4\tau)^7/\eta(2\tau)^7\eta(8\tau)^2,
  &\text{if the sum is congruent to }n/2\text{ modulo }n.
  \end{cases}
$$
Then this function $g$ is a modular unit on $X_0(N)$. The rest of
proof is the same as before and is omitted.

\section*{Acknowledgement}
The first author was partially supported by the National
ral Science Foundation of China (11801424), the Fundamental Research
Funds for the Central Universities (Project No. 2042018kf0027, Grant
1301–413000053) and a start- up research grant (1301–413100048) of the
Wuhan University. The second author was partially supported by Grant
106-2115-M-002-009-MY3 of the Ministry of Science and Technology,
Taiwan (R.O.C.).

\end{document}